\documentclass[12pt, reqno]{amsart}
\usepackage{amsmath,amsthm,amsfonts,
amssymb,latexsym,mathrsfs,bm,tikz,booktabs}
\usepackage{graphicx}
\usepackage{ifpdf,epstopdf}
\usepackage{float}
\usepackage{caption}

\usepackage[hypertexnames=false]{hyperref}
\hypersetup{
colorlinks=true,
linkcolor=red,
filecolor=brown,
citecolor=red
}
\newtheorem{theorem}{Theorem}[section]
\newtheorem{lemma}[theorem]{Lemma}
\newtheorem{prop}[theorem]{Proposition}

\theoremstyle{definition}

\theoremstyle{remark}
\newtheorem{remark}[theorem]{Remark}

\numberwithin{equation}{section}

\def\N{\mathbb{N}}

\def\Z{\mathbb{Z}}
\def\S{\mathfrak{S}}

\def\u{\mathbf{u}}

\newcommand{\des}{{\rm des}}
\newcommand{\maj}{{\rm maj}}
\newcommand{\inv}{{\rm inv}}

\newcommand{\cyc}{{\rm cyc}}

\newcommand{\rlmin}{{\rm rlmin}}
\newcommand{\Rlmin}{{\rm Rlmin}}

\def\cpk{{\rm{cpk}}}
\def\cval{{\rm{cval}}}
\def\cda{{\rm{cdrise}}}
\def\cdd{{\rm{cdfall}}}
\def\crise{{\rm{crise}}}
\def\cfall{{\rm{cfall}}}
\def\fix{{\rm{fix}}}
\def\pk{{\rm{pk}}}
\def\val{{\rm{val}}}
\def\da{{\rm{da}}}
\def\dd{{\rm{dd}}}
\def\fp{{\rm{fp}}}
\def\lin{{\rm{lin}}}
\def\ppk{{\rm{ppk}}}
\def\pval{{\rm{pval}}}
\def\pda{{\rm{pda}}}
\def\pdd{{\rm{pdd}}}

\def\Des{\mathop{\rm Des}\nolimits}
\def\Inv{\mathop{\rm Inv}\nolimits}
\newcommand\LD{\bm{LD}}

\newcommand{\bea}{\begin{array}}
\newcommand{\eea}{\end{array}}
\newcommand{\bg}{\begin{gathered}}
\newcommand{\eg}{\end{gathered}}
\newcommand{\be}{\begin{equation}}
\newcommand{\ee}{\end{equation}}
\def\bea#1\eea{\begin{align}#1\end{align}}

\begin{document}

\title[Mahonian-Stirling statistic for partial permutations]
{Mahonian-Stirling statistics for partial permutations}
\date{\today}
\author{Ming-Jian Ding}
\address[Ming-Jian Ding]{School of Mathematic Sciences,
 Dalian University of Technology, Dalian 116024, P. R. China}
\email{ding-mj@hotmail.com}

\author[Jiang Zeng]{Jiang Zeng}
\address[Jiang Zeng]{Univ Lyon, Universit\'e Claude Bernard Lyon 1, CNRS UMR 5208,
 Institut Camille Jordan, 43 blvd. du 11 novembre 1918, F-69622 Villeurbanne cedex, France}
\email{zeng@math.univ-lyon1.fr}





\dedicatory{}

\begin{abstract}
Recently Cheng et al. (Adv. in Appl. Math. 143 (2023) 102451) generalized the inversion number to partial permutations, which are also known as Laguerre digraphs, and asked for a suitable analogue of MacMahon's major index. 
We provide such a major index, namely,  
the corresponding maj and inv statistics are equidistributed, 
and exhibit a Haglund-Remmel-Wilson type identity.
We then interpret some Jacobi-Rogers polynomials 
in terms of Laguerre digraphs generalizing Deb and Sokal's alternating Laguerre digraph interpretation of some special Jacobi-Rogers polynomials.
\end{abstract}

\maketitle

\section{Introduction}
Let ${\bm{\gamma}}=(\gamma_0, \gamma_1, \ldots)$ and ${\bm{\beta}}=(\beta_1, \beta_2,\ldots)$ be two sequences of combinatorial numbers or polynomials  with $\beta_k\neq 0$ for all $k$.
The {\bf Jacobi-Rogers polynomials} $\mu_n:=\mu_n( \bm{\beta},\bm{\gamma})$
are the Taylor coefficients of the generic J-fraction
\begin{align}\label{J-CF}
   \sum_{n=0}^\infty \mu_n t^n
 = \cfrac{1}{1-\gamma_0 t-\cfrac{\beta_1 t^2}{1-\gamma_1 t-\cfrac{\beta_2t^2}{1-\cdots}}}.
\end{align}
The {\bf generalized Jacobi-Rogers polynomials}
$\mu_{n,k}:=\mu_{n,k}(\bm{\beta},\bm{\gamma})$ are defined by the recurrence 
\be\label{rec+catalan+matrix}
\begin{cases}
\mu_{0,0}=1, \quad \mu_{n,k}=0 \quad (k\notin [0, n]),\\
\mu_{n,k}=\mu_{n-1, k-1}+\gamma_k \mu_{n-1,k}+\beta_{k+1} \mu_{n-1, k+1} \quad  (n\geq k\geq  0).
\end{cases}
\ee
It is folklore~\cite{Fl80, GJ83, Vi83, Ai07, DS23} that $\mu_{n,k}$ is the generating function of the Motzkin paths from $(0,0)$ to $(n,k)$,
in which each rise gets weight 1, each fall from height $i$ gets weight $\beta_i$, 
and each level step at level $i$ gets weight $\gamma_i$.
In particular, the Jacobi-Rogers polynomial $\mu_{n}$ is the enumerative polynomial of Motzkin paths from $(0,0)$ to $(n,0)$.

The infinity lower triangular matrix $J=[\mu_{n,k}]_{n,k\geq 0}$ is known as
\emph{generalized Catalan  matrix}, \emph{Stieltjes tableau}, \emph{Catalan-Stieltjes tableau}, \emph{Jacobi-Rogers matrix}~\cite{Ai07, IZ10, PZ16, DS23}.
The inverse matrix $A=J^{-1}=[a_{n,k}]_{n,k\geq0}$ is also lower triangular with diagonal 1, 
and the row polynomials $p_n(x)=\sum_{k=0}^n a_{n,k} x^k$ satisfy the recurrence
\be\label{three-term-rec}
\begin{cases}
p_{-1}(x)=0, \quad p_{0}(x)=1,\\
p_{n+1}(x)=(x-\gamma_n) p_n(x)-\beta_n p_{n-1}(x)\quad (n\geq 0).
\end{cases}
\ee
By Favard's theorem the polynomial sequence $(p_n(x))_{n\geq0}$ is orthogonal with respect to 
the linear operator $\mathcal{L}: \Z[\bm{\beta},\bm{\gamma}][x]\to \Z[\bm{\beta},\bm{\gamma}]$ defined by $\mathcal{L}(x^n):=\mu_{n}$, and
\begin{equation*}
\mathcal{L}(p_m(x)p_n(x)):=\beta_0\beta_1\cdots \beta_n \delta_{m,n} \quad (\beta_0=1).
\end{equation*}
In this regard, the sequence $(\mu_{n})_{n\geq0}$ is called the \emph{moment sequence} of the orthogonal polynomials $p_n(x)'s$.
Let 
$$
P=(1, p_1(x), \ldots, p_n(x), \ldots)^T \; \text{and}\;X=(1, x, \ldots, x^n, \ldots)^T.
$$
Then, recurrence \eqref{three-term-rec} is equivalent to
\begin{equation}
P=A\, X\Longleftrightarrow X=J\,P,
\end{equation}
that is,
\begin{equation}\label{inverse matrix}
p_n(x)=\sum_{k=0}^n a_{n,k} x^k\Longleftrightarrow x^n=\sum_{k=0}^n \mu_{n,k}\,p_k(x).
\end{equation}

The  above scheme had its genesis   probably in 
   Euler's J-fraction~\cite{SZ22}
   \begin{subequations}
 \begin{equation}
  \sum_{n=0}^{\infty} n!\,x^n
   = \cfrac{1}{1-\gamma_0 t-\cfrac{\beta_1 t^2}{1-\gamma_1 t-\cfrac{\beta_2t^2}{1-\cdots}}}
 \end{equation}
with $\gamma_k=2k+1$ and $\beta_{k+1}=(k+1)^2$ for $k\geq 0$.
The corresponding orthogonal polynomials are the monic (simple) {\bf Laguerre polynomials} $L_n(x)$:
\begin{equation}
L_n(x)=\sum_{k=0}^n (-1)^{n-k}\binom{n}{k}\frac{n!}{k!} x^k
\Longleftrightarrow
x^n= \sum_{k=0}^n \binom{n}{k} \frac{n!}{k!} L_k(x),
 \end{equation}
and the  generalized Jacobi-Rogers polynomials are  positive integers
\begin{equation}
\mu_{n,k}=\binom{n}{k} \frac{n!}{k!}.
\end{equation}
\end{subequations}

Let $[n]:=\{1,2,\ldots,n\}$.
A \emph{\bf Laguerre configuration} on $[n]$ is a pair $\pi=(I,f)$,
where $I$ is a subset of $[n]$ and $f: I\to [n]$ is an injection, and
is called a $k$-Laguerre configuration on $[n]$ if $|I|=n-k$. 
One can represent a Laguerre configuration by one of the following models~\cite{FS82, DS23, CEH23}.
\begin{enumerate}
\item
A {\bf Laguerre word} over the alphabet $[n]\cup \{\lozenge\}$ is a word consisting of distinct letters in $[n]$
and the symbol $\lozenge$.
Each Laguerre configuration corresponds to a string of $n$ symbols $\pi_1\pi_2\cdots \pi_n$,
some of which are distinct numbers in the range from 1 to $n$ and the remaining ones of
which are some special ``hole'' symbols $\lozenge$.
In this formulation, the domain $I$ of the Laguerre configuration consists with 
the positions in the string that do not contain a hole,
and each such position is mapped to the number in that position. Clearly, a $k$-Laguerre configuration on $[n]$ corresponds 
to a Laguerre word on $[n]$ with $k$ symbol.s 
$\lozenge$.
For example, word $\pi=3\,2\, 5\, \lozenge\,1\,8\,6\,\lozenge$ is a 2-Laguerre word over $[8]$.
\item
A {\bf Laguerre digraph} on $[n]$ is a directed graph with vertex set $[n]$ such that
each vertex has 0 or 1 in-degree and 0 or 1 out-degree. It is clear that there are two types of connected components in a Laguerre digraph: cycles and paths.
A  $k$-Laguerre configuration $(I, f)$ 
corresponds to a Laguerre digraph 
on $[n]$ with $k$ paths such that 
there is an arrow $i\to f(i)$ if and only if  $i\in I$.
For the above Laguerrre word $\pi$  the corresponding Laguerre digraph is depicted in Figure~\ref{graph_tran}.
\begin{figure}[ht]
\begin{tikzpicture}[thick]

\draw [->]  (4.65,2.00)--(4.65,1.00);
\draw [->]  (4.65,1.15)--(4.65,-.05);

\coordinate [label=90:$7$] (7) at (4.45,1.60);
\coordinate [label=90:$6$] (6) at (4.45,0.80);
\coordinate [label=90:$8$] (8) at (4.45,-.10);

\foreach \x in {4.65} \foreach \y in {-0.10,1.00,2.05}
\fill (\x,\y)circle (.05);


\draw[->] (7.20,1.00) arc (0:116:0.8);
\fill (7.20,1.00) circle (.05);
\coordinate [label=90:$5$] (5) at (7.40,0.75);

\draw[->] (6.00,1.70) arc (120:236:0.8);
\fill (6.00,1.70) circle (.05);
\coordinate [label=90:$1$] (1) at (6.00,1.70);

\draw[->] (6.00,0.30) arc (240:356:0.8);
\fill (6.00,0.30) circle (.05);
\coordinate [label=90:$3$] (3) at (5.90,-0.20);

\draw[->] (9.00,1.00) arc (0:356:0.4);
\fill (9.00,1.00) circle (.05);
\coordinate [label=90:$2$] (2) at (9.20,0.75);

\fill (10.00,1.00) circle (.05);
\coordinate [label=90:$4$] (4) at (10.20,0.75);
\end{tikzpicture}
\caption{Laguerre digraph for $\pi=3\,2\, 5\, \lozenge\,1\,8\,6\,\lozenge$}\label{graph_tran}
\end{figure}

\item
A {\bf partial permutation} on $[n]$ is a bijection between two subsets of $[n]$. This can be described by a $n\times n$ $(0,1)$-matrix
such that each row and column has at most one 1. Clearly, 
a $k$-Laguerre word $\pi_1\pi_2\cdots \pi_n$ on $[n]$ can be encoded by a (0,1)-matrix $[a_{i,j}]_{i,j=1}^{n}$ with $a_{i,j}=1$
if and only if $j=\pi_i$. Hence a 
$k$-Laguerre word corresponds to a $(0,1)$-matrix with $k$ 0-rows  and $k$ 0-columns.
\end{enumerate}

Let 
$\S_{n}^k$ be the set of all partial permutations on $[n]$ with $k$ 0-rows (and $k$ 0-columns).
It is clear that  the cardinality of $\S_{n}^k$ is $\binom{n}{k} \frac{n!}{k!}$ with $0\leq k\leq n$ and 
$\S_{n}^{0}=\S_n$, i.e., the set of all permutations of $[n]$.
In what follows, we shall mainly use Laguerre digraphs and  words
to describe statistics over Laguerre configurations or partial permutations.

%

\medskip

For $n\in \N$, we define the $(\beta,q)$-analogue of $n$ by
$$
[n]_{\beta, q}=\beta-1+q^0+q+\cdots +q^{n-1}, \quad
[0]_{\beta,q}=0,
$$
and
$[n]_{\beta,q}! := [1]_{\beta, q}[2]_{\beta, q} \cdots[n]_{\beta, q}$ with $[0]_{\beta,q}!=1$.
When $\beta=1$ we write $[n]_q :=[n]_{1, q}$ and $[n]_q! := [n]_{1, q}!$.
The $q$-binomial coefficient is defined  by
\begin{equation*}
   {n\brack k}_{q} := \frac{[n]_q!}{[k]_q![n-k]_q!}\quad\text{for}\quad 0 \leq k \leq n.
\end{equation*}

The following 
 $(\beta,q)$-analogue of Euler's J-fraction expansion \eqref{J-CF} is known 
 \begin{subequations}
 \begin{equation}
 \sum_{n=0}^{\infty} [n]_{\beta, q}!\,x^n=
  \cfrac{1}{1-\gamma_0 t-\cfrac{\beta_1 t^2}{1-\gamma_1 t-\cfrac{\beta_2t^2}{1-\cdots}}} 
 \end{equation}
  with
$
\gamma_k=q^k([k]_q+[k+1]_{\beta, q})$, $\beta_{k+1}=q^{2k+1}[k+1]_q [k+1]_{\beta,q}$ for $ k\geq 0$.
The corresponding orthogonal polynomials $L_n(x;a|q)$ are the normalized monic little $q$-{\bf Laguerre polynomials}
$p_n(x;a|q)$ with $a=\frac{1-q}{(1-q)\beta+q}$, see~\cite{ AGRZ99,DeVi94, KLS10},
namely,
\begin{align}
L_n(x;a|q)&= \frac{(aq/(1-q);q)_n}{a^n}(-1)^n q^{{n\choose 2}}p_n(ax;a/(1-q)|q)\nonumber\\
          &= \sum_{k=0}^n {n\brack k}_q (-1)^{n-k} q^{\binom{n-k}{2}} a^{k-n} \frac{(aq/(1-q);q)_n}{(aq(1-q);q)_k} x^k\nonumber\\
          &= \sum_{k=0}^n {n\brack k}_q (-1)^{n-k} q^{\binom{n-k}{2}}[k+1]_{\beta,q}\cdots [n]_{\beta,q}x^k.
\end{align}
The inversion formula reads
\begin{align}\label{inverse:q-Laguerre}
x^n=\sum_{k=0}^n {n\brack k}_q [k+1]_{\beta,q}\cdots [n]_{\beta,q} L_k(x;a|q).
\end{align}
\end{subequations}

The combinatorics of classical ($q$-)Laguerre  polynomials and their moments  are well studied~\cite{FS82, Vi83, DeVi94, SS96, IZ10, KS15, PZ20,SZ22}.  In two recent papers,   
Cheng et al. \cite{CEH23} and Deb and Sokal~\cite{DS23} 
generalized several well-known permutation statistics
to partial permutations or Laguerre digraphs. 
In particular, Cheng et al. \cite{CEH23} gave a combinatorial interpretation of the $q$-coefficients in \eqref{inverse:q-Laguerre}  
using an inversion-like statistic 
$\bm{\inv}$ on partial permutations
and raised the problem~\cite[Problem 8.1]{CEH23} to find a \emph{suitable} major index on partial permutations, while Deb and Sokal~\cite{DS23} extended cycle-alternating permutations to cycle-alternating Laguerre digraphs and interpreted some Jacobi-Rogers polynomials in terms of Laguerre digraphs.
In this paper we shall provide such a major index $\bm{\maj}$, which is deemed appropriate 
because it is equidistributed with the inversion statistic $\bm{\inv}$, 
see Theorem~\ref{thm+gf+maj+rlmin+partial} and Theorem~\ref{thm+main+inv+maj+ide}. We then  interpret some 
 Jacobi-Rogers polynomials in terms of Laguerre digraphs generalizing  Deb and Sokal's alternating Laguerre digraph interpretation 
of some special Jacobi-Rogers polynomials, see Theorem~\ref{thm+gen+Jac+Rog+Laguerre}.

The rest of this paper is organized as follows, we will present our main results in Section~\ref{sec+main+results}.
The first one is a major index $\maj$ on partial permutations 
such that the bi-statistic $(\maj, \rlmin)$ of 
major index and right-to-left minimum is equidistributed with the bi-statistic $(\inv, \rlmin)$ of inversion number and right-to-left minimum 
on partial permutations, the corresponding generating function is also got, see Theorem \ref{thm+gf+maj+rlmin+partial}.
The second one is to generalize the equidistribution of inversion number and major index over partial permutations to a Haglund-Remmel-Wilson type identity for partial permutations, see Theorem \ref{thm+main+inv+maj+ide}.
The last one is to extend cyclic statistics on permutations to partial permutations and evaluate the generalized Jacobi-Rogers polynomials, see Theorem \ref{thm+gen+Jac+Rog+Laguerre}.
In Section \ref{sec+gf+maj+equi}, we give a proof of Theorem \ref{thm+gf+maj+rlmin+partial} by a bijection,
and prove Theorem \ref{thm+main+inv+maj+ide}, which gives an answer to Problem 8.1 in \cite{CEH23}.
Lastly, a proof of Theorem \ref{thm+gen+Jac+Rog+Laguerre} will be given in Section \ref{sec+eulerian+partial}.

\section{Main results}\label{sec+main+results}

Let $M$ be the multiset $\{1^{m_1},2^{m_2},\ldots,r^{m_r}\}$,
where ${m_i}$ is the number of copies of letter $i$ and $n=m_1+m_2+\cdots+m_r$.
Denote by $\S_M$ the set of all permutations of set $M$.
If all $m_i$ are equal to 1, then $\S_M$ is the symmetric group $\S_n$.
Let $\pi \in \S_M$  be a permutation with $\pi=\pi_1\pi_2\cdots \pi_n$.
Define the set of all inversion indexes ending at position $i$ of $\pi$ by
\begin{equation}\label{lehmer+code}
\Inv^{\Box,i}(\pi):=\{j\,|\,1 \leq j < i\leq n, \pi_j > \pi_i \}
\end{equation}
and denote by $\inv^{\Box,i}(\pi)$ the cardinality of set $\Inv^{\Box,i}(\pi)$,
thus the \emph{inversion number} of $\pi$ is given by
\begin{equation*}
\inv(\pi):=\sum_{i=1}^n\inv^{\Box,i}(\pi).
\end{equation*}
Define the descent set of $\pi$  by
\begin{equation*}
\Des(\pi):=\{ i \in [n-1]~|~ \pi_i > \pi_{i+1} \}
\end{equation*}
and the number of descents of $\pi$ is the cardinality of $\Des(\pi)$, denoted by $\des(\pi)$.
The \emph{major index} \cite{Mac15} of $\pi$ is defined as
\begin{equation*}
  \maj(\pi) := \sum_{i \in \Des(\pi)}i.
\end{equation*}

Let $(A,B)$ be an ordered partition of $[n]$ with $|A|=n-k$.
Define the between-set inversion $\inv(A,B)=|\{(a,b)\in A\times B\,|\, a>b\}|$.
It is known \cite[Lemma 2.6.6 (2)]{GJ83} that
\begin{align}\label{gaussian}
	\sum_{A\uplus B=[n], |A|=n-k}q^{\inv(A,B)}={n\brack k}_q.
\end{align}

In what follows, we shall consider an element $\pi \in \S_{n}^k$ as a permutation
$\pi_1\pi_2\cdots \pi_n$ of the $n$-subset of multiset $\{1, 2, \ldots, n, \lozenge^k\}_<$,
where the letters are ordered as~$1<2<\cdots <n<\lozenge$
and $\pi$ has $k$ ``hole" symbols $\lozenge$.
A \emph{descent} of a partial permutation $\pi \in \S_{n}^k$ is an index $i\in [n-1]$ 
such that $\pi_i > \pi_{i+1}$, and the descent set of $\pi$ is defined by
\begin{equation*}
	\Des(\pi):=\{i \in [n-1]\,|\,\pi_i > \pi_{i+1} \}.
\end{equation*}
Let $\des(\pi)=|\Des(\pi)|$.
For each $\pi =(I,f)\in \S_{n}^k$,
let $S_{\pi}=\{f(i)\,|\,i \in I\}$ and $\overline{S}_{\pi}=[n]\backslash S_{\pi}$.
Define the inversion number of $\pi$ by
\begin{equation}\label{def:inv}
	\inv(\pi)=\inv_0(\pi) +\inv(S_{\pi},\overline{S}_{\pi}),
\end{equation}
where $\inv_0(\pi) :=\sum_{i=1}^n\inv^{\Box,i}(\pi)$.
We define the \emph{major index} of $\pi$ by
\begin{equation}\label{def+maj+partial}
	\maj(\pi):=\maj_0(\pi)+\inv(S_{\pi},\overline{S}_{\pi}),
\end{equation}
where $\maj_0(\pi)=\sum_{i \in \Des(\pi)} i$.

Given $\pi \in \S_n^k$, we say that letter $\pi_i \in [n]$ is a \emph{right-to-left minimum}
if all letters in $[\pi_{i}-1]$ are in front of $\pi_i$ in $\pi$.
Note that $1$ is always a right-to-left minimum.
Denote by $\Rlmin(\pi)$ the set of all right-to-left minima in $\pi$
and $\rlmin(\pi)$ its  cardinality.
We show that the bi-statistic $(\maj,\rlmin)$ is equidistributed with $(\inv,\rlmin)$ 
and their generating function has a product formula.

\begin{theorem}\label{thm+gf+maj+rlmin+partial}
Let $n,k$ be nonnegative integers. Then
\begin{equation}\label{Id+gf+maj+rlmin+partial}
   \sum_{\pi \in \S_{n}^{k}}\beta^{\rlmin(\pi)}q^{\maj(\pi)}
 = \sum_{\pi \in \S_{n}^{k}}\beta^{\rlmin(\pi)}q^{\inv(\pi)}={n\brack k}_q [k+1]_{\beta,q}\cdots [n]_{\beta,q}.
\end{equation}
\end{theorem}
The second identity in \eqref{Id+gf+maj+rlmin+partial} is due to Cheng et al.~\cite{CEH23}.
When $\beta=1$, we have
\begin{equation}\label{eq+gf+maj}
 \sum_{\pi \in \S_{n}^{k}}q^{\inv(\pi)}= \sum_{\pi\in\S_{n}^k}q^{\maj(\pi)}={n\brack k}_{q}\frac{[n]_q!}{[k]_q!},
\end{equation}
which gives an answer to the open problem in \cite[Problem 8.1]{CEH23}.
When $k=0$, Remmel and Wilson \cite{RW15} generalized
the first equality of \eqref{eq+gf+maj}
to  the following identity
\begin{equation}\label{identity+HRW}
     \sum_{\pi \in \S_n}q^{\inv(\pi)}\prod_{i \in \Des(\pi)}\left(1+\frac{z}{q^{1+\inv^{\Box,i}(\pi)}}\right)
   = \sum_{\pi \in \S_n}q^{\maj(\pi)}\prod_{i=1}^{\des(\pi)}\left(1+\frac{z}{q^{i}} \right),
\end{equation}
which had been conjectured by Haglund.
Obviously, the $z=0$ case of \eqref{identity+HRW} is the first equality in \eqref{eq+gf+maj} with $k=0$. 
Actually, MacMahon~\cite{Mac15} (see also Andrews \cite[Theorem 3.6, Theorem 3.7]{And76})
has proved that the statistics inversion number and major index are equidistributed on $\S_M$
and their common generating function has a product formula,
\begin{equation}\label{ide+equi+mac}
\sum_{\pi \in \S_M}q^{\inv(\pi)} = \sum_{\pi \in \S_M}q^{\maj(\pi)}=\frac{[n]_q!}{[m_1]_q![m_2]_q!\cdots[m_r]_q!}.
\end{equation}
In 2016 Wilson \cite{Wi16} extended identity \eqref{identity+HRW} to multiset $M$,
\begin{equation}\label{ide+HRW+multiset}
     \sum_{\pi \in \S_M}q^{\inv(\pi)}\prod_{i \in \Des(\pi)}\left(1+\frac{z}{q^{1+\inv^{\Box,i}(\pi)}}\right)
   = \sum_{\pi \in \S_M}q^{\maj(\pi)}\prod_{i=1}^{\des(\pi)}\left(1+\frac{z}{q^{i}} \right).
\end{equation}
If $m_i=1$ for all $i$, then identity \eqref{ide+HRW+multiset} reduces to \eqref{identity+HRW},
and if $z=0$, then identity \eqref{ide+HRW+multiset} reduces to \eqref{ide+equi+mac}.
Since then, an identity analogous to \eqref{identity+HRW} connecting major-like and inv-like stststics
is called \emph{Haglund-Remmel-Wilson identity}~\cite{Wi16, Liu23Eur, Liu23DM, YLYZ23}.

The following theorem is our second  main result,
which gives an answer to the open problem in \cite[Problem 8.1]{CEH23} when $z=0$.
%


\begin{theorem}\label{thm+main+inv+maj+ide}
For nonnegative integers $n,k$, we have
\begin{equation}\label{main+ide+RHW+partial}
     \sum_{\pi \in \S_{n}^k}q^{\inv(\pi)}\prod_{i \in \Des(\pi)}\left(1+\frac{z}{q^{1+\inv^{\Box,i}(\pi)}}\right)
   = \sum_{\pi \in \S_{n}^k}q^{\maj(\pi)}\prod_{i=1}^{\des(\pi)}\left(1+\frac{z}{q^{i}} \right).
\end{equation}
\end{theorem}

Let $\LD_{n,k}$ be the set of all Laguerre digraphs on $[n]$ with $k$ paths.
 In order to define suitable statistics on Laguerre digraphs, as in~\cite{Ze93, DS23},
we need to make a convention about boundary conditions at the two-ends of a path.
The idea is to extend a Laguerre graph $G$ on the vertex set $[n]$ to a digraph $\widehat{G}$
on the vertex set $[n]\cup \{0\}$ by decreeing
that any vertex $i\in [n]$ that has in-degree (resp. out-degree) 0 in $G$
will receive an incoming (resp. outgoing) edge from (resp. to) the vertex 0.
In this way, each vertex $i\in [n]$ will have a unique predecessor $p(i)\in [n]\cup \{0\}$
and a unique successor $s(i)\in [n]\cup \{0\}$.
We then say that a vertex $i\in [n]$ is a
\begin{enumerate}
\item a peak ($\pk$), if $p(i)<i>s(i)$;
\item a valley ($\val$), if $p(i)>i<s(i)$;
\item a double ascent ($\da$), if $p(i)<i<s(i)$;
\item a double descent ($\dd$), if $p(i)>i>s(i)$;
\item a fixed point ($\fp$), if $p(i)=i=s(i)$.
\end{enumerate}

We write $\pk(G)$, $\val(G)$, $\da(G)$, $\dd(G)$, $\fp(G)$ for the number of vertices $i\in [n]$
that are, respectively, peaks, valleys, double ascents, double descents and fixed points in Laguerre digraph $G$.
We also write $\cyc(G)$ for the number of cycles in Laguerre digraph~$G$.

\begin{theorem}\label{thm+gen+Jac+Rog+Laguerre}
The generalized Jacobi-Rogers polynomial $\mu_{n,k}(\bm{\beta},\bm{\gamma})$ with
\begin{align}
\gamma_k=k(u_3+u_4)+\alpha\beta, \quad \beta_k=k(\beta-1+k)u_1u_2
\end{align}
is the following enumerative polynomial of Laguerre digraphs on $[n]$ with $k$ paths
\be\label{jacobi-laguerre}
\mu_{n,k}(\bm{\beta},\bm{\gamma})=\sum_{G\in \LD_{n,k}}u_1^{\pk(G)-k}u_2^{\val(G)}u_3^{\da(G)}u_4^{\dd(G)}\alpha^{\fp(G)}\beta^{\cyc(G)}.
\ee
\end{theorem}

\begin{remark}
An \emph{alternating Laguerre digraph} is  a Laguerre digraph in which there are no double ascents, double descents or fixed points.  Specializing 
Theorem~\ref{thm+gen+Jac+Rog+Laguerre} to 
$\alpha=u_3=u_4=0$ and $u_1=u_2=1$, we derive  that 
the Jacobi-Rogers polynomial $\mu_{n,k}(\bm{\beta},\bm{\gamma})$ with 
$\beta_k=k(\beta-1+k)$ and $\bm{\gamma=0}$
enumerates alternating Laguerre digraphs on $[n]$  with $k$ paths with a weight $\beta$ for each cycle~\cite[Proposition~6.3]{DS23}.
\end{remark}

\section{Mahonian statistics  over partial permutations}\label{sec+gf+maj+equi}

For $\pi \in \S_{n}^k$ let $M_\pi$ be its corresponding (0,1)-matrix. 
Originally Cheng et al. \cite{CEH23} defined the inversion number ($\widetilde\inv$) 
of $M_\pi$ as the number of \emph{survival} 0's in $M_\pi$ after deleting
\begin{itemize}
\item all 0's below 1 in each column,
\item all 0's at the right of 1 in each row,
\item all 0's if there is no 1 in neither their rows nor their columns.
\end{itemize}
Similarly
we define the major index ($\widetilde{\maj}$) of $M_\pi$ as the number of  0's such that 
\begin{itemize}
\item  
 they are above  the 1 (of their columns)  which has  another  1 
 at right in the upper row,  or 
 \item  
 they are above  the 1 (of their columns)  which has  no   1 
in the upper row, or
 \item  
they are at left of  a 1 of  their rows (if their columns have no 1).
\end{itemize}
For example, if $\pi=3\,2\, 5\, \lozenge\,1\,8\,6\,\lozenge\in \S_8^{2}$, then
$$
\Des(\pi)=\{1,4,6\},\quad S_{\pi}=\{1,2,3,5,6,8\}, \quad \overline{S}_{\pi}=\{4,7\}.
$$
Hence
$\inv_0(\pi)=8$, $\inv(S_\pi, \overline{S}_{\pi})=4$, and
\[\inv(\pi)=12,\quad \textrm{and}\quad \maj(\pi)=(1+4+6)+4=15.
\]
The following is the (0,1)-matrix $M_{\pi}$ of $\pi$, where the survival 0's are replaced by $x$ 
in the left and right matrices in \eqref{01-matrix} for the computation of $\inv(\pi)=12$ 
and $\maj(\pi)=15$ in terms of (0,1)-matrix $M_\pi$, respectively.

\begin{equation}\label{01-matrix}
\left[
  \begin{array}{cccccccc}
    \mathbf{x} & \mathbf{x} & 1 & 0  & 0 & 0 & 0  & 0 \\
    \mathbf{x} & 1 & 0 & 0 & 0 & 0 & 0 & 0 \\
    \mathbf{x} & 0  & 0 & \mathbf{x}& 1 & 0 & 0 & 0 \\
    \mathbf{x} & 0 & 0 & 0  & 0 & \mathbf{x}& 0 & \mathbf{x}\\
    1 & 0 & 0 & 0  & 0 & 0 & 0 & 0 \\
    0 & 0 & 0 & \mathbf{x} & 0 & \mathbf{x} & \mathbf{x} & 1 \\
    0 & 0 & 0 & \mathbf{x} & 0 & 1 & 0  & 0 \\
    0  & 0  & 0  & 0  & 0  & 0  & 0  & 0  \\
  \end{array}
\right]=
\left[
  \begin{array}{cccccccc}
    \mathbf{x} & \mathbf{x} & 1 & 0 & 0 & \mathbf{x} & 0 & 0 \\
    \mathbf{x}& 1 & 0 & 0& 0 & \mathbf{x}  & 0& 0 \\
   \mathbf{x}& 0 & 0 & \mathbf{x} & 1 & \mathbf{x}  & 0& 0 \\
   \mathbf{x} & 0  & 0  & 0  & 0 & \mathbf{x}   & 0  & 0  \\
    1 & 0 & 0 & 0 & 0 & \mathbf{x} & 0 & 0 \\
    0 & 0 & 0 &\mathbf{x} & 0 & \mathbf{x}  & \mathbf{x} & 1 \\
    0 & 0 & 0 & \mathbf{x} & 0 & 1 & 0 & 0 \\
     0  & 0  & 0 & 0  & 0  & 0  & 0  & 0  \\
  \end{array}
\right].
\end{equation}

The following result shows that the two latter definitions of inversion number 
and major index in terms of (0,1)-matrix are equivalent to the definitions on Laguerre words.
\begin{prop}\label{prop+inv+word}
For $\pi \in \S_{n}^k$, we have
\begin{equation}\label{eq:inv}
\inv(\pi)=\widetilde{\inv}(\pi),\quad \maj(\pi)=\widetilde{\maj}(\pi).
\end{equation}
\end{prop}
\begin{proof} 
We just prove the first identity for inversion numbers because the second can be proved similarly.
Let $M_\pi$ be the (0,1)-matrix associated to $\pi \in \S_{n}^k$.
There are two cases for a column having a survival 0.
\begin{itemize}
\item
If a survival 0 in the $j$th column  has a 1 below, say $\pi_i=j$,
then this 0 must be in the $\ell$th row with $\ell<i$ and $j$th column of $M_{\pi}$ with a 1 at the right ($\pi_\ell> \pi_i$)
or without 1 in the $\ell$th row ($\pi_\ell=\lozenge$), i.e., $\pi_\ell>\pi_i$,
this corresponds to an inversion $\pi_{\ell} > \pi_i$ and $1 \leq \ell < i$ in $\Inv^{\Box,i}(\pi)$
and vice versa.
\item
If a survival 0 in the $j$th column does not have a 1 below,
then there must be a 1 at the right.
Assume that this 0 is at position $(i,j)$ of $M_{\pi}$,
then $j \in \overline{S}_{\pi}$ and there exists some $\ell > j$ such that $\pi_i=\ell$ and $\ell \in S_{\pi}$.
Therefore, this 0 enumerates an inversion in $\Inv(S_{\pi},\overline{S}_{\pi})$ and vice versa.
\end{itemize}
Combining the above cases, the proof is complete.
\end{proof}

Given $\pi \in \S_{n}^{k}$, define the set of all inversion indexes starting at a letter of $\pi$
which is different from $\lozenge$ as
\begin{equation}\label{lehmer+code+excluding+lozenge}
\Inv_{\blacklozenge}(\pi):=\{(i,j)\,|\,1 \leq i < j\leq n, \pi_i > \pi_j, \pi_i \neq \lozenge \}
\end{equation}
and denote by $\inv_\blacklozenge(\pi)$ the cardinality of set $\Inv_\blacklozenge(\pi)$.
Similarly, define the set of all inversion indexes starting at symbols $\lozenge$ of $\pi$ as
\begin{equation}\label{lehmer+code+including+lozenge}
\Inv_{\lozenge}(\pi):=\{(i,j)\,|\,1 \leq i < j\leq n, \pi_i > \pi_j, \pi_i = \lozenge \}
\end{equation}
and denote by $\inv_{\lozenge}(\pi)$ the cardinality of set $\Inv_{\lozenge}(\pi)$.
Define
$\widetilde\inv_{\blacklozenge}(\pi):=\inv_{\blacklozenge}(\pi)+\inv(S_{\pi},\overline{S}_{\pi})$.
By \eqref{def:inv}, we have
\begin{equation}\label{iden+inv+black+white}
\inv(\pi)=\widetilde\inv_{\blacklozenge}(\pi)+\inv_{\lozenge}(\pi).
\end{equation}

Let $\binom{[n]}{k}$ denote the set of all subsets of $k$ elements of $[n]$. 
For $I\in \binom{[n]}{k}$
let
$$
\widetilde\S_{n}^{k}(I)=\{\pi\in \S_n^k\, |\, \pi_i=\lozenge \;\; \text{for} \;\; i \in I\}.
$$
%
For example, if $n=3$ and $I=\{2\}$,
then $$\widetilde\S_{3}^{1}(\{2\})=\{1\lozenge2, 1\lozenge3,2\lozenge3,2\lozenge1,3\lozenge1,3\lozenge2\}.$$
Let $w(\pi)=\beta^{\rlmin(\pi)}q^{\widetilde\inv_{\blacklozenge}(\pi)}$. 
Then the corresponding terms are as follows.
{\small
\begin{center}
\vspace{-0.8em}
\begin{table}[ht]\caption{The weights $w(\pi)$ in $\widetilde\S_{3}^{1}(\{2\})$}
\vspace{-0.5em}
\begin{tabular}{ccccccc} 
\toprule[1.5pt]
$\pi$ & $1\lozenge 2$ & $1\lozenge 3$ & $2\lozenge 3$ & $2\lozenge 1$&$3\lozenge 1$&
$3\lozenge 2$
\\ 
\midrule[0.5pt]
\rule{0pt}{10pt}
$S_\pi \backslash \overline{S}_\pi$  & $\{1,\,2\}\backslash \{3\}$ & $\{1,\,3\}\backslash  \{2\}$  & $\{2,\,3\}\backslash  \{1\}$ &$\{1,2\} \backslash \{3\}$
&$\{1,3\} \backslash \{2\}$&$\{2,3\} \backslash \{1\}$\\ 
\rule{0pt}{10pt}
$w(\pi)$  & $\beta^2$ & $\beta q$  & $q^2$&$\beta q$ &$\beta q^2$&$q^3$ \\ 
\bottomrule[1.5pt]
\vspace{-3.0em}
\end{tabular}
\label{table:nonlin}
\end{table}
\end{center}
}
Hence
$$
\sum_{\pi \in \widetilde\S_{3}^{1}(\{2\})}\beta^{\rlmin(\pi)}q^{\widetilde\inv_{\blacklozenge}(\pi)}=(\beta+q)(\beta+q+q^2).
$$

The following result was proved in~\cite[Theorem 3.2]{CEH23} using (0,1)-matrix.
For completeness we provide a proof using Laguerre~words.
\begin{lemma}\label{lem+partial+rlmin+inv+black}
 We have
\begin{align}
    \sum_{\pi \in \S_{n}^{k}}\beta^{\rlmin(\pi)}q^{\inv(\pi)}
 = {n\brack k}_q [k+1]_{\beta,q}\cdots [n]_{\beta,q}.
\end{align}
\end{lemma}
\begin{proof} 
For any $I\in \binom{[n]}{k}$
let $\sigma=\sigma_1\sigma_2\cdots\sigma_n$ 
with $\sigma_{i}=\lozenge$ if $i \in I$ and empty otherwise.
Recall that
\be\label{inv-tilde}
\widetilde\inv_{\blacklozenge}(\pi):=\inv_{\blacklozenge}(\pi)+\inv(S_{\pi},\overline{S}_{\pi}).
\ee
We make up a Laguerre word $\pi \in \widetilde\S_n^k(I)$ by choosing $j_1, j_2,\ldots, j_{n-k}$, successively,
in $[n]$ and putting them at the feasible empty positions in $\sigma$ from left.
Let $I_0=[n]$, and assume that $j_{\ell}$ is the $j_\ell^*$-th smallest element 
in $I_{\ell-1}=[n]\backslash \{j_1, j_2, \ldots, j_{\ell-1}\}$,
at step $\ell$ ($1\leq \ell\leq n-k$) choosing $j_\ell$ in the $\ell$-th feasible empty position from left results in
sending the $j_\ell^*-1$ smaller elements in $I_{\ell-1}$ to the right of $j_\ell$ or in $\overline{S}_\pi$,
so its contribution to $\widetilde\inv_{\blacklozenge}(\pi)$ is $j_\ell^*-1$;
also $j_\ell$ is a right-to-left minimum if and only if $j_\ell=1^*$, i.e., it is the minimum in $I_{\ell-1}$. 
Thus, the contribution of this insertion to the weight of $\pi$ is $\beta+q+\cdots +q^{n-\ell}$ 
because there are $n-\ell+1$ choices for $j_\ell$ in $I_{\ell-1}$.
Thus we obtain the generating polynomial
\begin{equation}\label{eq+cheng}
   \sum_{\pi \in \widetilde\S_n^k(I)}\beta^{\rlmin(\pi)}q^{\inv_{\blacklozenge}(\pi)}
 = \prod_{\ell=1}^{n-k}(\beta+q+\cdots +q^{n-\ell}),
\end{equation}
which depends only on the size $k$ of  $I\in \binom{[n]}{k}$.
For $\pi\in  \widetilde\S_n^k(I)$ let $w_I=w_1w_2\cdots w_n$ with $w_i=1$ if $i\in I$ and 0 otherwise. 
From \eqref{lehmer+code+including+lozenge} 
it is clear that $\inv_{\lozenge}(\pi)=\inv(w_I)=\inv(\overline{I}, I)$ with $\overline{I}=[n]\backslash I$.
Therefore,
\begin{align*}
    \sum_{\pi \in \S_{n}^{k}}\beta^{\rlmin(\pi)}q^{\inv(\pi)}
 &= \sum_{I\in \binom{[n]}{k}}\sum_{\pi \in \widetilde\S_n^k(I)}\beta^{\rlmin(\pi)}q^{\widetilde\inv_{\blacklozenge}(\pi)+\inv_{\lozenge}(\pi)}\\
 &= \sum_{I\in \binom{[n]}{k}}q^{\inv(\overline{I}, I)}\sum_{\pi \in \widetilde\S_n^k(I)}\beta^{\rlmin(\pi)}q^{\widetilde\inv_{\blacklozenge}(\pi)}\\
 &= {n\brack k}_q [k+1]_{\beta,q}\cdots [n]_{\beta,q},
\end{align*}
where the last line follows from \eqref{gaussian} and \eqref{eq+cheng}.
\end{proof}

We need some more definitions and preliminaries.
Recall that $\S_M$ is the all permutations of multiset $M=\{1^{m_1},2^{m_2},\ldots,r^{m_r}\}$
with $m_1+m_2+\cdots +m_r=n$ and $m_1,m_2, \ldots, m_r\geq 0$.
For any $\pi \in \S_M$ with $\pi=\pi_1\pi_2\cdots \pi_n$,
we label the $m_i$ copies of each letter $i\in [r]$ from left to right by the order of their occurrence as $i_1,i_2, \ldots, i_{m_i}$
and let $\pi^*$ be the nondecreasing rearrangement of the letters in $\pi$
such that the equal letters are ordered by their labels,
$$
\pi^*=\pi_1^*\pi_2^*\cdots \pi_n^*=1_1\cdots 1_{m_1} 2_1\cdots 2_{m_2}\cdots r_1\cdots r_{m_r}.
$$
Let $I(\pi):=(b_1,b_2,\ldots,b_n)$,
where $b_i$ is the number of letters $\pi_j$ to the right of $\pi^*_i$ in $\pi$ and $\pi^*_i > \pi_j$.
Clearly, we have $b_i \leq i-1$ and
\begin{equation}\label{b-code:inv}
\inv(\pi)=b_1+b_2+\cdots+b_n.
\end{equation}
For example, if we take $\pi=2_1\,1_1\,2_2\,6_1\,5_1\,4_1\,4_2\,3_1$, then 
$\pi^*=1_1\,2_1\,2_2\,3_1\,4_1\,4_2\,5_1\,6_1$,
$I(\pi)=(0,1,0,0,1,1,3,4)$ and $\inv(\pi)=10$.

\subsection*{Carlitz's insertion algorithm}

We now define a mapping $\psi$ on $\S_M$ recursively.  
This mapping was first used by Carlitz~\cite{Ca75} in $\S_n$ and extended to $\S_M$ in \cite{FH08, Wi16}.

Let $\pi \in \S_M$ with $I(\pi) = (b_1, b_2,\ldots, b_n)$.
For $i\in [n]$, one construct the words $\alpha_i$ on $M^{i}=\{\pi_1^*,\pi_2^*,\ldots,\pi^*_{i}\}$
starting with $\alpha_1=\pi^*_1$. 
At step $i+1$, we label the $i+1$ slots in $\alpha_i$:
in front of $\alpha_i$, between two consecutive letters of $\alpha_i$, 
and at the end of $\alpha_i$, according to the following rule:
\begin{enumerate}
\item Label the rightmost slot in $\alpha_i$ with 0;
\item Label its descent slot from {\bf right to left} with $1,2,\ldots,\des(\alpha_i)$;
\item Label the leftmost slot with $\des(\alpha_i)+1$;
\item Label the other (non-descent) slots from the left to right with $\des(\alpha_i)+2,\des(\alpha_i)+3,\ldots,i$.
\end{enumerate}
Inserting $\pi^*_{i+1}$ in the slot labelled $b_{i+1}$ in $\alpha_i$ creates a permutation $\alpha_{i+1}$.
Repeat the process until $M^n$ and define $\psi(\pi):=\alpha_n$.
It is easy to verify that $\psi$ is a bijection satisfying
\[
\maj(\alpha_{i+1})-\maj(\alpha_i)=b_{i+1},\qquad 1\leq i\leq n-1.
\]
Thus
\begin{equation}\label{b-code:maj}
\maj(\psi(\pi))=b_1+b_2+\cdots + b_n=\inv(\pi).
\end{equation}
We omit the proof, and refer the reader to \cite[Section 2.2]{Wi16} and \cite{FH08} for more~details.

For our running example $\pi=21265443$, we have
\begin{equation*}
\alpha_4= {}_{2}2_11_32_43_0,
\end{equation*}
with $\pi^*_5=4$ and $b_5=1$.
Inserting $4$ at slot 1 in $\alpha_4$, we get $\alpha_5=24123$ and finally $\psi(\pi)=26544123$.

%

Let $\pi \in \S_M$ with $\pi=\pi_1\pi_2 \cdots \pi_n$.
A letter $\pi_i$ is called a \emph{\bf right-to-left minimum}
if $\pi_i \leq \pi_j$ for all $j > i$ and $\pi_i$ is the first occurrence (from left) in $\pi$.
Denote by $\Rlmin(\pi)$ the set of all right-to-left minima in $\pi$
and $\rlmin(\pi)$ the cardinality of $\Rlmin(\pi)$.
For instance, if $\pi=21265443$, then $\psi(\pi))=26544123$
and
\[\Rlmin(\pi)=\{1,3\}=\Rlmin(\psi(\pi)).
\]

\begin{theorem}\label{thm+rlm+maj+inv+multiset}
Let  $M=\{1^{m_1},2^{m_2},\ldots,r^{m_r}\}$.
The mapping $\psi$ is a bijection on $\S_M$ such that
\begin{subequations}\label{stirling-macmahon}
\begin{align}
\Rlmin (\pi)&=\Rlmin(\psi(\pi)),\label{stirling-macmahon1}\\
(\inv,\rlmin)(\pi)&=(\maj,\rlmin)(\psi(\pi)).\label{stirling-macmahon2}
\end{align}
\end{subequations}
\end{theorem}
\begin{proof} By \eqref{b-code:inv} and \eqref{b-code:maj}
we need only to verify \eqref{stirling-macmahon1}.
For any $\pi \in \S_M$ with $\pi^*=\pi_1^*\pi_2^*\cdots \pi_n^*$ and $I(\pi)=(b_1,b_2, \ldots, b_n)$, by definition,
a letter $\pi^*_i$ is in $\Rlmin(\pi)$ ($i\in [n]$) if and only if $b_i=0$ 
and $\pi^*_i$ is the first occurrence from left in $\pi$.

Assume that $\pi^*_i\in \Rlmin(\pi)$, then $b_i=0$,
by definition of mapping $\psi$, we have $\alpha_i=\alpha_{i-1}\pi^*_i$,
so $\pi^*_i$ is a right-to-left minimum in $\alpha_j$ for $j\geq i$
because $\pi^*_j\geq \pi^*_i$ and  $b_j=0$ if $\pi^*_j=\pi_i^*$, 
which implies that $\pi^*_i$ is the first occurrence in $\psi(\pi)$.
Hence $\pi^*_i \in \Rlmin(\psi(\pi))$.

Assume that $\pi^*_i \in \Rlmin(\psi(\pi))$, then $\pi^*_i$ is the first occurrence in $\psi(\pi)$ and $b_i=0$.
If $b_i \geq 1$, then there would be a letter $\gamma<\pi_i^*$ at the right of $\pi^*_i$ in $\alpha_i$, 
and then in all $\alpha_j$ with $j\geq i$, which contradicts the assumption.
According to the action of $\psi$, if $\pi^*_i$ is not the first occurrence in $\pi$, 
then $b_i \geq 1$ because $\pi^*_i$ is the first occurrence in $\psi(\pi)$.


Summarizing, we have proved that $\pi^*_i$ is the first occurrence in $\pi$ and $b_i=0$, 
that is, $\pi^*_i \in \Rlmin(\pi)$.
\end{proof}

\begin{remark}
Identity~\eqref{stirling-macmahon} is known for  $\S_n$, see \cite[Theorem 5.6]{BW91}.
\end{remark}

\begin{proof}[\bf Proof of Theorem \ref{thm+gf+maj+rlmin+partial}]
For  any $J\in \binom{[n]}{n-k}$ let
\begin{equation*}
\S_n^k(J)=\{\pi \in \S_{n}^k\,|\, S_\pi=J\}.
\end{equation*}
Thus
\begin{align*}
    \sum_{\pi \in \S_{n}^{k}}\beta^{\rlmin(\pi)}q^{\inv(\pi)}
 &= \sum_{J\in \binom{[n]}{n-k}} q^{\inv(J,\overline{J})} \sum_{\pi \in \S_n^k(J)}\beta^{\rlmin(\pi)}q^{\inv_0(\pi)}, \\
    \sum_{\pi \in \S_{n}^{k}}\beta^{\rlmin(\pi)}q^{\maj(\pi)}
 &= \sum_{J\in \binom{[n]}{n-k}} q^{\inv(J,\overline{J})}
 \sum_{\pi \in \S_n^k(J)}\beta^{\rlmin(\pi)}q^{\maj_0(\pi)}.
\end{align*}
For any $\pi \in \S_n^k(J)$, we show that the bijection $\psi$ in Theorem~\ref{thm+rlm+maj+inv+multiset} has the property
\begin{align}
\inv_0(\pi)&=\maj_0(\psi(\pi)),\\
\quad \rlmin(\pi)&=\rlmin(\psi(\pi)).
\end{align}
Let $i$ be the largest integer in $[n]$ such that $[i] \subseteq J$.
Clearly $\Rlmin(\pi) \subseteq [i]$ by definition.
Thus, we need only to consider the restriction of $\pi$ on $[i]$, denoted $\pi|_{[i]}$, 
which preserves  the order of the elements of $[i]$ in $\pi$ (as 3\,2\,1 in $3\,2\, 5\, \lozenge\,1\,8\,6\,\lozenge$).
So, statistic $\rlmin(\psi(\pi))$ only depends on $\pi|_{[i]}$.
Applying Theorem \ref{thm+rlm+maj+inv+multiset} with $M=[i]$, the result is immediate. 
The proof is completed by Lemma~\ref{lem+partial+rlmin+inv+black}.
\end{proof}

\begin{proof}[\bf Proof of Theorem \ref{thm+main+inv+maj+ide}]
Let $J_0=[n-k]$. It is clear that the order preserving mapping from $J_0$ to
any $J\in \binom{[n]}{n-k}$ leads to a bijection from $\S_n^k(J_0)$ to $\S_n^k(J)$,
which keeps the traces of the integer valued statistics $\inv_0$, $\maj_0$ and $\inv^{\Box,i}$,
and the set valued statistic $\Des$.
Therefore,
combining with \eqref{gaussian} we have
\begin{align}\label{eq+iden+inv+dec}
  &\quad \sum_{\pi \in \S_{n}^k}q^{\inv(\pi)}\prod_{i \in \Des(\pi)}
     \left(1+\frac{z}{q^{1+\inv^{\Box,i}(\pi)}}\right) \nonumber \\
 &= \sum_{J\in \binom{[n]}{n-k}}q^{\inv(J,\overline{J})}\sum_{\pi \in \S_n^k(J)}q^{\inv_0(\pi)}
    \prod_{i \in \Des(\pi)}\left(1+\frac{z}{q^{1+\inv^{\Box,i}(\pi)}}\right)\nonumber\\
 &= {n\brack k}_q \sum_{\pi\in \S_n^k(J_0)}q^{\inv_0(\pi)}\prod_{i \in \Des(\pi)}\left(1+\frac{z}{q^{1+\inv^{\Box,i}(\pi)}}\right).
\end{align}
Similarly we derive
\begin{equation*}
    \sum_{\pi \in \S_{n}^k(J)}q^{\maj(\pi)}\prod_{i=1}^{\des(\pi)}\left(1+\frac{z}{q^{i}} \right)
  = \sum_{\pi\in \S_n^k(J_0)}q^{\maj_0(\pi)}\prod_{i=1}^{\des(\pi)}\left(1+\frac{z}{q^{i}} \right),
\end{equation*}
and then
\begin{align}\label{eq+iden+maj+dec}
 \quad \sum_{\pi \in \S_{n}^k}q^{\maj(\pi)}\prod_{i=1}^{\des(\pi)}\left(1+\frac{z}{q^{i}} \right)
 =  {n\brack k}_q \sum_{\pi\in \S_n^k(J_0)}q^{\maj_0(\pi)}\prod_{i=1}^{\des(\pi)}\left(1+\frac{z}{q^{i}} \right).
\end{align}
 Thus, it suffices  to verify the following identity
\begin{equation}\label{eq+iden+inv+maj+dec}
    \sum_{\pi\in \S_n^k(J_0)}q^{\inv_0(\pi)}\prod_{i \in \Des(\pi)}\left(1+\frac{z}{q^{1+\inv^{\Box,i}(\pi)}}\right)
  = \sum_{\pi\in \S_n^k(J_0)}q^{\maj_0(\pi)}\prod_{i=1}^{\des(\pi)}\left(1+\frac{z}{q^{i}} \right),
\end{equation}
which follows from identity \eqref{ide+HRW+multiset} with the multiset $M=\{1, \ldots, n-k,\lozenge^k\}$.
\end{proof}

\section{Eulerian statistic over Laguerre digraphs}\label{sec+eulerian+partial}

Let $\S_n$ denote the set of all permutations of $[n]$.
Given a permutation $\pi\in \S_n$, each value $i$, $1\leq i\leq n$,
can be classified according to one of the following five cases:
\begin{enumerate}
\item a cycle peak ($\cpk$), if $\pi^{-1}(i)<i>\pi(i)$;
\item a cycle valley ($\cval$), if $\pi^{-1}(i)>i<\pi(i)$;
\item a cycle double rise ($\cda$), if $\pi^{-1}(i)<i<\pi(i)$;
\item a cycle double fall ($\cdd$), if $\pi^{-1}(i)>i>\pi(i)$;
\item a fixed point ($\fix$), if $\pi(i)=i$.
\end{enumerate}
We write $\cpk(\pi)$, $\cval(\pi)$, $\cda(\pi)$, $\cdd(\pi)$, $\fix(\pi)$ for the number of values $i \in [n]$
that are, respectively, cycle peaks, cycle valleys, cycle double rises, cycle double falls and fixed points in permutation $\pi$.
It is easy to see that $\cpk(\pi)=\cval(\pi)$.
Define the refinement of Eulerian polynomials 
\begin{align}\label{poly:zeng}
    P^{\textrm{cyc}}_n(\u, \alpha, \beta)
 := \sum_{\pi\in \S_n}u_1^{\cpk(\pi)}u_2^{\cval(\pi)}u_3^{\cda(\pi)}u_4^{\cdd(\pi)}\alpha^{\fix(\pi)}\beta^{\cyc(\pi)}
\end{align}
with $\u=(u_1,u_2, u_3,u_4)$, and the exponential generating 
function~\cite{Ze93} 
\be
    G^{\textrm{cyc}}(z)
 := 1+\sum_{n=1}^{\infty}P^{\textrm{cyc}}_n(\u,\alpha, \beta)\frac{z^n}{n!}.
\ee

\begin{lemma}\cite[Th\'eor\`eme 1]{Ze93}\label{lem-fix-carlitz-scoville} 
We have
\begin{align}\label{fix-carlitz-scoville}
   G^{\cyc}(z)
 = e^{\alpha\beta z}\left(\frac{\alpha_1-\alpha_2}{\alpha_1 e^{\alpha_2z}-\alpha_2 e^{\alpha_1 z}}\right)^\beta,
\end{align}
where $\alpha_1\alpha_2=u_1u_2$ and $\alpha_1+\alpha_2=u_3+u_4$.
\end{lemma}

\begin{lemma}\cite[Th\'eor\`eme 3]{Ze93}\label{lem+J-CF} 
We have
\begin{align}\label{fix-carlitz-scoville-cf}
   1+\sum_{n=1}^\infty P^{\cyc}_n(\u,\alpha,\beta)z^n
 = \cfrac{1}{1-\gamma_0 z-\cfrac{\beta_1z^2}{1-\gamma_1z-\cfrac{\beta_2z^2}{1-\gamma_2z-\cdots}}},
\end{align}
where
$\gamma_k=k(u_3+u_4)+\alpha\beta$ and $\beta_{k+1}=(k+\beta)(k+1)u_1u_2$.
\end{lemma}
Define the enumerative polynomial of permutations with linear statistics 
and 0-0 boundary conditions by
\begin{equation*}
P_n^{\lin}(\u):=\sum_{\pi\in \S_n} u_1^{\ppk(\pi)}u_2^{\pval(\pi)}u_3^{\pda(\pi)}u_4^{\pdd(\pi)},
\end{equation*}
where $\ppk(\pi), \pval(\pi), \pda(\pi)$ and $\pdd(\pi)$ denote, respectively, 
the number of path peaks, path valleys, path double ascents and path double descent in the permutation $\pi$ written as a word $\pi_1\pi_2\cdots \pi_n$,
where we impose the boundary conditions $\pi_0=\pi_{n+1}=0$.
\begin{lemma}\cite[Proposition 4 and (2.5)]{Ze93}\label{4.3}
We have
\begin{equation*}
  G^{\lin}(z):=\sum_{n=1}^\infty P_n^{\lin}(\u)\frac{z^n}{n!}
 = u_1 \biggl(\frac{e^{\alpha_2 z}-e^{\alpha_1z}}{\alpha_2 e^{\alpha_1 z}-\alpha_1 e^{\alpha_2 z}}\biggr),
\end{equation*}
where $\alpha_1\alpha_2=u_1u_2$ and $\alpha_1+\alpha_2=u_3+u_4$.
Moreover,
\begin{equation}\label{eq-diff}
(G^{\lin}(z))'=u_1+(u_3+u_4)G^{\lin}(z)+u_2(G^{\lin}(z))^2.
\end{equation}
\end{lemma}

Let $J_k(z)$ be the exponential generating function of the $k$-th column of 
the Catalan matrix $J=[\mu_{n,k}( \bm{\beta},\bm{\gamma})]_{n,k\geq0}$ (see \eqref{rec+catalan+matrix}),
that is
\begin{equation}
   J_k(z)=\sum_{n=k}^{\infty} \mu_{n,k}(\bm{\beta},\bm{\gamma})\frac{z^n}{n!}
 = \frac{z^k}{k!}+\mu_{k+1,k}( \bm{\beta},\bm{\gamma})\frac{z^{k+1}}{(k+1)!}+\cdots,
\end{equation}
in particular,
\begin{equation*}
J_0(z)=\sum_{n=0}^{\infty} \mu_{n}(\bm{\beta},\bm{\gamma}))\frac{z^n}{n!}
\end{equation*}
is the exponential generating function of the moments $\mu_n's$.
%
%

\begin{lemma}\cite[p. 310]{Ai07}
There exists a function $F(z)=\sum_{n=0}^{\infty} F_n \frac{z^n}{n!}$ with~$F(0)=0$ such that
\be\label{eq:sheffer}
J_k(z)=J_0(z)\frac{(F(z))^k}{k!} \;\; \textrm{for all}\;\; k
\ee
if and only if $\gamma_k=a+ks$, $\beta_k=k(b+(k-1)u)$ and
\be\label{Aigner:system}
\begin{cases}
F'(z)=1+sF(z)+u(F(z))^2,\\
J_0'(z)=aJ_0(z)+bJ_0(z)F(z).
\end{cases}
\ee
\end{lemma}

\begin{proof}[\bf Proof of Theorem~\ref{thm+gen+Jac+Rog+Laguerre}]
We define the multivariate generating function of \emph{Laguerre digraphs}
\be\label{poly:laguerre digraph}
   \widehat{L}_{n,k}(\u, \alpha, \beta)
 := \sum_{G\in \LD_{n,k}}u_1^{\pk(G)}u_2^{\val(G)}u_3^{\da(G)}u_4^{\dd(G)}\alpha^{\fp(G)}\beta^{\cyc(G)}.
\ee
A Laguerre digraph $G\in \LD_{n,k}$ consists of a permutation (that is a collection of disjoint cycles)
on some subset $S\subseteq [n]$ together with $k$ disjoint paths on $[n]\backslash S$.
Each of these paths can be considered as a permutation written in word form.
Hence, the exponential generating function of $\widehat{L}_{n,k}(\u, \alpha, \beta)$ is the product of $G^{\textrm{cyc}}(z)$ and 
${(G^{\lin }(z))^k}/{k!}$, i.e.,
\begin{equation}\label{egf+LD}
   \sum_{n=k}^{\infty}\widehat{L}_{n,k}(\u, \alpha, \beta)\frac{z^n}{n!}
 = G^{\textrm{cyc}}(z)\frac{(G^{\lin }(z))^k}{k!}.
\end{equation}
On the other hand, define the Jacobi-Rogers matrix $J=[\mu_{n,k}(\bm{\beta},\bm{\gamma}))]_{n,k\geq0}$ with
$$
\gamma_k=k(u_3+u_4)+\alpha\beta,\quad \beta_{k+1}=(k+\beta)(k+1)u_1u_2.
$$
It follows from  Lemma~\ref{lem-fix-carlitz-scoville} and Lemma~\ref{lem+J-CF}
that
$J_0(z):=G^{\textrm{cyc}}(z)$. Then, 
by  Lemma~\ref{lem-fix-carlitz-scoville} and Lemma~\ref{4.3} we  see that $F(z)=G^{\lin }(z)/u_1$ is the solution of system~\eqref{Aigner:system} with
\begin{equation*}
a=\alpha\beta,\; b=\beta u_1u_2,\; s=u_3+u_4,\; u=u_1u_2.
\end{equation*}
Comparing with \eqref{eq:sheffer} and  \eqref{egf+LD} we obtain \eqref{jacobi-laguerre}.
\end{proof}

\begin{remark} The generating function \eqref{egf+LD} is equivalent to  \cite[(6.11)]{DS23}.
 Zhu~\cite{Zhu20} defined
the triangular array 
 $[T_{n,k}]_{n,k\geq 0}$ by the recurrence relation
\begin{align*}
T_{n,k}=&\lambda (a_1k+a_2)T_{n-1,k-1}\\
&+[(b_1-da_1)n-(b_1-2da_1)k+b-2-d(a_1-a_2)]T_{n-1,k}\nonumber\\
                     &+\frac{d(b_1-da_1)}{\lambda}(n-k+1)T_{n-1,k+1},
\end{align*}
where $T_{0,0}=1$ and $T_{n,k}=0$ unless $0\leq k\leq n$.  It is shown~\cite{Zhu20} that  this array and the corresponding polynomials $T_n(x)=\sum_{k=0}^n T_{n,k}x^k$ have properties similar to 
 the classical Eulerian numbers and polynomials, 
and  the generating function
$\sum_{n=0}^{\infty} T_n(x) t^n$ has the J-fraction expansion \eqref{J-CF}, see  \cite[Theorem 2.7]{Zhu20},  with
\begin{subequations}\label{zhu:cf}
\begin{align}
  \gamma_k &=(ka_1+a_2)(\lambda+dx)+(kb_1+b_2)x,\\
\beta_{k+1}&=(k+1)\bigl((ka_1+a_2)b_1+a_1b_2\bigr)x(\lambda+dx).
\end{align}
\end{subequations}
However, a combinatorial interpretation of $T_n(x)$  in its full generality is missing.
In fact,  such a combinatorial interpretation can be derived from Lemma~\ref{lem+J-CF}  as in the following.
Applying the following substitution in \eqref{fix-carlitz-scoville-cf}, 
\begin{subequations}
\begin{align*}
&u_1=u_3=a_1(\lambda+dx),\quad
u_2=u_4=b_1x,\\
&\alpha={\bigl(a_1b_1\bigl(a_2(\lambda+dx)+b_2x\bigr)\bigr)}/{(a_1b_2+a_2b_1)},\quad \beta={a_2}/{a_1}+{b_2}/{b_1},
\end{align*}
\end{subequations}
and comparing with \eqref{zhu:cf},  we obtain the following 
 combinatorial interpretation of 
 polynomials $T_n(x)$:
\begin{align}
T_n(x)&=\sum_{\pi\in \S_n}a_1^{\crise(\pi)-\cyc_2(\pi)}b_1^{\cfall(\pi)-\cyc_2(\pi)}x^{\cfall(\pi) }\nonumber\\
      &\qquad \times(\lambda+dx)^{\crise(\pi)}\bigl(a_2(\lambda+dx)+b_2x\bigr)^{\fix(\pi)}(a_2b_1+a_1b_2)^{\cyc_2(\pi)},
\end{align}
where $\cyc_2(\pi)=\cyc(\pi)-\fix(\pi)$,
and 
$\crise(\pi), \cfall(\pi)$ and $\fix(\pi)$ denote, respectively,
the number of cycle rises ($i < \pi(i)$), cycle falls ($i > \pi(i)$) and fixed points ($i=\pi(i)$) in permutation $\pi$.
2\end{remark}
%
\section*{Acknowledgements}

The first author was supported by the \emph{China Scholarship Council} (No. 202206060088).
This work was done during his visit at Universit\'e Claude Bernard Lyon 1 in 2022-2023.


\end{document}